\documentclass[a4paper,12pt]{amsart}
\usepackage{graphicx}
\usepackage{amscd}

\newcommand{\comment}[1]{\marginpar{\sffamily{\noindent\tiny #1
   \par}\normalfont}}
\renewcommand{\comment}[1]{}

\newcommand{\C}{{\mathbb C}}

\newcommand{\R}{{\mathbb R}}

\newtheorem{theorem}{Theorem}[section]

\newtheorem{corollary}[theorem]{Corollary}
\theoremstyle{definition}

\newtheorem{example}[theorem]{Example}
\newtheorem{remark}[theorem]{Remark}

\begin{document}
\title[On normal embedding of complex algebraic surfaces ]
{On normal embedding of complex algebraic surfaces}
\author{Lev Birbrair}
\address{Departamento de Matem\'atica, Universidade Federal do Cear\'a
(UFC), Campus do Pici, Bloco 914, Cep. 60455-760. Fortaleza-Ce,
Brasil} \email{birb@ufc.br}
\author{Alexandre Fernandes}
\address{Departamento de Matem\'atica, Universidade Federal do Cear\'a
(UFC), Campus do Pici, Bloco 914, Cep. 60455-760. Fortaleza-Ce,
Brasil} \email{alexandre.fernandes@ufc.br}
\author{Walter D.
  Neumann}
\address{Department of Mathematics, Barnard College,
  Columbia University, New York, NY 10027}
\email{neumann@math.columbia.edu}


\begin{abstract} We construct examples of complex algebraic surfaces
  not admitting normal embeddings (in the sense of semialgebraic or
  subanalytic sets) with image a complex algebraic surface.
\end{abstract}

\maketitle

{Dedicated to our friends Maria (Cidinha) Ruas and Terry Gaffney in
connection to their 60-th birthdays.}

\section{Introduction}

Given a closed and connected subanalytic subset $X\subset\R^m$ the
\emph{inner metric} $d_X(x_1,x_2)$ on $X$ is defined as the infimum of
the lengths of rectifiable paths on $X$ connecting $x_1$ to $x_2$.
Clearly this metric defines the same topology on $X$ as the Euclidean
metric on $\R^m$ restricted to $X$ (also called \emph{``outer
  metric''}).  But the inner metric is not necessarily bi-Lipschitz
equivalent to the Euclidean metric on $X$. To see this it is enough to
consider a simple real cusp $x^2=y^3$. A subanalytic set is called
\emph{normally embedded} if these two metrics (inner and Euclidean)
are bi-Lipschitz equivalent.

\begin{theorem}[\cite{BM}] Let $X\subset\R^m$ be a connected and
globally subanalytic set. Then there exist a normally embedded
globally subanalytic set $\tilde{X}\subset\R^q$ and a global
subanalytic homeomorphism $p\colon\tilde{X}\rightarrow X$
bi-Lipschitz with respect to the inner metric. The pair
$(\tilde{X},p)$, is called a normal embedding of $X$.
\end{theorem}

The original version of this theorem (see \cite{BM}) was formulated
in a semialgebraic language, but it easy to see that this result
remains true for a global subanalytic structure or, moreover, for
any o-minimal structure. The proof remains the same as in \cite{BM}

Complex algebraic sets and real algebraic sets are globally
subanalytic sets. By the above theorem these sets admit globally
subanalytic normal embeddings. Tadeusz Mostowski asked if there exists
a complex algebraic normal embedding when $X$ is complex algebraic
set, i.e., a normal embedding for which the image set
$\tilde{X}\subset\C^n$ is a complex algebraic set. In this note we
give a negative answer for the question of Mostowski. Namely, we prove
that a Brieskorn surface $x^b+y^b+z^a=0$ does not admit a complex
algebraic normal embedding if $b>a$ and $a$ is not a divisor of $b$. For
the proof of this theorem we use the ideas of the remarkable paper of
A. Bernig and A. Lytchak \cite{BL} on metric tangent cones and the
paper of the authors on the $(b,b,a)$ Brieskorn surfaces \cite{BFN2}.
We also briefly describe other examples based on taut singularities.

\subsection*{Acknowledgements}
The authors acknowledge research support under the grants: CNPq
grant no 301025/2007-0 (Lev Birbriar),  FUNCAP/PPP 9492/06, CNPq
grant no 300393/2005-9 (Alexandre Fernandes) and  NSF grant no.\
  DMS-0456227 (Walter Neumann).

\section{Proof}

Recall that a subanalytic set $X\subset\R^n$ is called
\emph{metrically conical} at a point $x_0$ if there exists an
Euclidean ball $B\subset\R^n$ centered at $x_0$ such that $X\cap B$
is bi-Lipschitz homeomorphic, with respect to the inner metric, to
the metric cone over its link at $x_0$. When such a bi-Lipschitz
homeomorphism is subanalytic we say that $X$ is
\emph{subanalytically metrically conical} at $x_0$.

\begin{example}{\rm The Brieskorn surfaces in $\C^3$}
$$\{(x,y,z) ~|~ x^b+y^b+z^a=0\}$$ {\rm ($b>a$) are subanalytically
metrically conical at $0\in\C^3$ (see \cite{BFN2}).}
\end{example}

We say that a complex algebraic set admits a \emph{complex algebraic
normal embedding} if the image of a subanalytic normal embedding of
this set can be chosen complex algebraic.

\begin{example} Any complex algebraic curve
admits a complex algebraic normal embedding. This follows from the
fact that the germ of an irreducible complex algebraic curve is
bi-Lipschitz homeomorphic with respect to the inner metric to the
germ of $\C$ at the origin.
\end{example}

\begin{theorem}\label{main_theorem} If $1<a<b$ and $a$ is not a
  divisor of $b$ then no neighborhood of $0$ in the
    Brieskorn surface in $\C^3$
$$\{(x,y,z)\in\C^3 ~|~ x^b+y^b+z^a=0\}$$
admits a complex algebraic normal embedding.
\end{theorem}

We will need the following result on
tangent cones.

\begin{theorem}\label{bernig_lytchak_2}
If $(X_1,x_1)$ and $(X_2,x_2)$ are germs of subanalytic sets which
are subanalytically bi-Lipschitz homeomorphic with respect to the
induced Euclidean metric, then their tangent cones $T_{x_1}X_1$ and
$T_{x_2}X_2$ are subanalytically bi-Lipschitz homeomorphic.
\end{theorem}

This result is a weaker version of the results of
Bernig-Lytchak(\cite{BL}, Remark 2.2 and Theorem 1.2). We present here
an independent proof.

\begin{proof}[Proof of Theorem \ref{bernig_lytchak_2}] Let us denote
$$S_xX=\{ v\in T_xX ~|~ |v|=1\}.$$ Since $T_xX$ is a cone over
$S_xX$, in order to prove that $T_{x_1}X_1$ and $T_{x_2}X_2$ are
subanalytically bi-Lipschitz homeomorphic, it is enough to prove
that $S_{x_1}X_1$ and $S_{x_2}X_2$ are subanalytically bi-Lipschitz
homeomorphic.

By Corollary 0.2 in \cite{V}, there exists a subanalytic
bi-Lipschitz homeomorphism with respect to the induced Euclidean
metric: $$h\colon (X_1,x_1)\rightarrow (X_2,x_2)\,,$$ such that
$|h(x)-x_2|=|x-x_1|$ for all $x$. Let us define $$dh\colon
S_{x_1}X_1\rightarrow S_{x_2}X_2$$ as follows: given $v\in
S_{x_1}X_1$, let $\gamma\colon[0,\epsilon)\rightarrow X_1$ be a
subanalytic arc such that
$$ |\gamma(t)-x_1|=t ~ \forall ~t\in [0,\epsilon)\quad
\mbox{and}\quad
\lim_{t\to 0^+}\frac{\gamma(t)-x_1}{t}=v\,;$$ we define
$$dh(v)=\lim_{t\to 0^+}\frac{h\circ\gamma(t)-x_2}{t}.$$ Clearly,
$dh$ is a subanalytic map. Define $d(h^{-1})\colon
S_{x_2}X_2\rightarrow S_{x_1}X_1$ the same way. Let $k>0$ be a
Lipschitz constant of $h$. Let us prove that $k$ is a Lipschitz
constant of $dh$. In fact, given $v_1,v_2\in S_{x_1}X_1$, let
$\gamma_1,\gamma_2\colon[0,\epsilon)\rightarrow X_1$ be subanalytic
arcs such that
$$ |\gamma_i(t)-x_1|=t ~ \forall ~t\in [0,\epsilon) \quad
\mbox{and} \quad \lim_{t\to 0^+}\frac{\gamma_i(t)-x_1}{t}=v \ i=1,2.$$
Then
\begin{eqnarray*}
|dh(v_1)-dh(v_2)| &=& \Bigl|\lim_{t\to 0^+}\frac{h\circ\gamma_1(t)-x_2}{t}-\lim_{t\to 0^+}\frac{h\circ\gamma_1(t)-x_2}{t}\Bigr| \\
&=& \lim_{t\to 0^+}\frac{1}{t}|h\circ\gamma_1(t)-h\circ\gamma_2(t)| \\
&\leq& k \lim_{t\to 0^+}\frac{1}{t}|\gamma_1(t)-\gamma_2(t)| \\
&=& k |v_1-v_2|.
\end{eqnarray*}
Since $d(h^{-1})$ is $k$--Lipschitz by the same argument and $dh$ and
$d(h^{-1})$ are mutual inverses, we have proved the theorem.
\end{proof}

\begin{corollary}\label{proposition6.2}
Let $X\subset\R^n$ be a normally embedded subanalytic set. If $X$ is
subanalytically metrically conical at a point $x\in X$, then the
germ $(X,x)$ is subanalytically bi-Lipschitz homeomorphic to the
germ $(T_xX,0)$.
\end{corollary}

\begin{proof} The tangent cone of the straight cone at the
vertex is the cone itself. So the corollary is a
direct application of Theorem \ref{bernig_lytchak_2}.
\end{proof}

\begin{proof}[Proof of the \ref{main_theorem}]
Let $X\subset\C^3$ be the complex algebraic surface defined by
$$X=\{(x,y,z) ~|~ x^b+y^b+z^a=0\}.$$ We are going to prove that the
germ $(X,0)$ does not have a normal embedding in $\C^N$ which is a
complex algebraic surface. In fact, if
$(\tilde{X},0)\subset(\C^N,0)$ is a complex algebraic normal
embedding of $(X,0)$ and $p\colon(\tilde{X},0)\rightarrow (X,0)$ is
a subanalytic bi-Lipschitz homeomorphism, since $(X,0)$ is
subanalytically metrically conical \cite{BFN2}, then
$(\tilde{X},{0})$ is subanalytically metrically conical and by
Corollary \ref{proposition6.2} $(\tilde{X},{0})$ is subanalytically
bi-Lipschitz homeomorphic to $(T_{0}\tilde{X},{0})$. Now, the
tangent cone $T_{0}\tilde{X}$ is a complex algebraic cone, thus its
link is a $S^1$-bundle. On the other hand, the link of $X$ at $0$ is
a Seifert fibered manifold with $b$ singular fibers of degree $\frac
a{\gcd(a,b)}$. This is a contradiction because the Seifert fibration
of a Seifert fibered manifold (other than a lens space) is unique up
to diffeomorphism.
\end{proof}

The following result relates the metric tangent cone of $X$ at $x$
and the usual tangent cone of the normally embedded sets. See
\cite{BL} for a definition of a metric tangent cone.

\begin{theorem}
[\cite{BL}, Section 5]\label{bernig_lytchak} Let $X\subset\R^m$ be a
closed and connected subanalytic set and $x\in X$. If
$(\tilde{X},p)$ is a normal embedding of $X$, then
$T_{p^{-1}(x)}\tilde{X}$ is bi-Lipschitz homeomorphic to the metric
tangent cone $\mathcal{T}_xX$.
\end{theorem}

\begin{remark}{\rm We
showed that the metric tangent cones of the above Brieskorn surface
singularities are not homeomorphic to any complex cone.}
\end{remark}

\subsection{Other examples} We sketch how taut surface singularities
% in the sense of Laufer \cite{taut}
give other examples of complex surface germs without any complex
algebraic normal embeddings.  We first outline the argument and then
give some clarification.

Both the inner metric and the outer (euclidean) metric on a complex
analytic germ $(V,p)$ are determined up to bi-Lipschitz equivalence by
the complex analytic structure (independent of a complex
embedding). This is because $(f_1,\dots,f_N)\colon
(V,p)\hookrightarrow (\C^N,0)$ is a complex analytic embedding if and
only if the $f_i$ generate the maximal ideal of $\mathcal O_{(V,p)}$,
% the local ring of $(V,p)$,
and adding to the set of generators gives an embedding that induces
the same metrics up to bi-Lipschitz equivalence.  A taut complex
surface germ is an algebraicly normal germ (to avoid confusion we say
``algebraicly normal'' for algebro-geometric concept of normality)
whose complex analytic structure is determined by its topology.  Thus
a taut singularity whose inner and outer metrics do not agree can have
no complex analytic normal embedding. Taut complex surface
singularities were classified by Laufer \cite{taut} and include, for
example, the simple singularities. The simple singularities of type
$B_n$, $D_n$, and $E_n$ have non-reduced tangent cones, from which
follows easily that they have non-equivalent inner and outer
metrics. Thus, they admit no complex algebraic normal embeddings.

There is an issue with this argument, in that we have restricted to
complex analytic embeddings of $(V,p)$, that is, embeddings that
induce an isomorphism on the local ring. But one can find holomorphic
maps that are topological embeddings but which only induce an
injection on the local ring (the image will no longer be an
algebraicly normal germ). Such a map is not a complex analytic
embedding, but it will still be holomorphic and real
semialgebraic. It is not hard to see that the non-reducedness of
the tangent cone persists when one restricts the local ring, so the
argument of the previous paragraph still applies.

\end{document}